\documentclass[a4paper,10pt]{article}
\usepackage{amsmath,amssymb,graphicx,hyperref,amsthm}

\newcommand{\ve}[1]{\mathbf{#1}}
\newcommand{\m}[1]{\mathcal{#1}}
\newcommand{\cF}{\mathcal{F}}
\newcommand{\cV}{\mathcal{V}}
\newcommand{\F}{\mathbb{F}}

\newcommand{\Fxv}{\F\left[\ve x\right]}
\newcommand{\gr}{Gr\"obner\ }

\DeclareMathOperator{\Sh}{Sh}

\newtheorem{thm}{\bfseries Theorem}
        
\newtheorem{prop}[thm]{\bfseries Proposition}
\newtheorem{cor}[thm]{\bfseries Corollary}

\begin{document}

\title{Standard monomials and extremal point sets\footnote{The results presented are the outcome of a very long research project. During this period the author was affiliated with several institutes: the project started at Budapest University of Technology and Economics (undergraduate studies); was continued at Central European University, Budapest (Phd studies); and final bits were added at Freie Universit\"at Berlin (postdoctoral researcher). This publication is the complete version of an extended abstract that appeared in the proceedings of EuroComb'17\cite{MR3}. Most of the results presented are also part either of the master thesis or of the PhD dissertation of the author~\cite{M1,M2}.}} 

\author{Tam\'as M\'esz\'aros\thanks{Freie Universit\"at Berlin. Supported by the Postdoctoral Fellow program of the Dahlem Research School. E-mail: tamas.meszaros@fu-berlin.de}}

\date{}

\maketitle

\begin{abstract}
We say that a set system $\cF\subseteq 2^{[n]}$ shatters a set $S\subseteq [n]$ if every possible subset of $S$ appears as the intersection of $S$ with some element of $\cF$ and we denote by $\Sh(\cF)$ the family of sets shattered by $\cF$. According to the Sauer-Shelah lemma we know that in general, every set system $\cF$ shatters at least $|\cF|$ sets and we call a set system shattering-extremal if $|\Sh(\cF)|=|\cF|$. In \cite{M1} and \cite{RM}, among other things, an algebraic characterization of shattering-extremality was given, which offered the possibility to generalize the notion to general finite point sets. Here we extend the results obtained for set systems to this more general setting, and as an application, strengthen a result of Li, Zhang and Dong from \cite{LZD}.
\end{abstract}

\begin{quote}
\end{quote}

\section{Introduction} 

\subsubsection*{Shattering-extremal families}

A set system $\cF\subseteq 2^{[n]}$ \emph{shatters} a given set $S \subseteq[n]$ if the elements of $\cF$ intersect it in every possible way, i.e. $2^S = \left\{F \cap S\; : \; F \in \cF  \right\}$. 
The family of subsets of $[n]$ shattered by $\cF$ is denoted by $\Sh(\cF)$.  The size of the largest set shattered by $\cF$ is called the Vapnik-Chervonenkis dimension of $\cF$, and is denoted by $dim_{VC}(\cF)$. Shattering and the
Vapnik-Chervonenkis dimension are present in several fields of both applied and theoretical
mathematics, among others in theoretical machine learning, combinatorics, statistics, probability theory, computer science and logic.

In general, maybe a bit surprisingly, we have that $|\Sh(\cF)|\geq |\cF|$ for every set system $\cF\subseteq 2^{[n]}$. This statement was proved by several authors independently, and is often referred to as the Sauer-Shelah lemma. For a nice simple proof see e.g.~\cite{ARS}. Set systems satisfying the Sauer-Shelah lemma with equality are called \emph{shattering-extremal}, or \emph{s-extremal} for short. They are well studied notions in mathematics with a rich history (see e.g.~\cite{ARS,BLR,BR,F2,KM,M1,M2,MR1,MR2,Mo,RM}), also appearing in the literature under the names \emph{ample classes} (see e.g.~\cite{BCDK,D}) and \emph{lopsided classes}(see e.g.~\cite{L}). For an easy example of an s-extremal family one can consider \emph{down-sets} (respectively \emph{up-sets}), i.e. families $\cF\subseteq 2^{[n]}$ with the property that whenever $H\subseteq F$ (respectively $F\subseteq H$) and $F\in \cF$ then we also have $H\in \cF$. Now if $\cF$, is a down-set (respectively up-set) then $\cF$ is s-extremal, simply because in this case  $\Sh(\cF)= \cF$ (respectively $\Sh(\cF)=\{[n]\setminus F:F\in \cF\}$).

\subsubsection*{Extremal point sets} 

Given some set $F\subseteq [n]$, let $v_F \in \{0,1 \}^n$ be its \emph{characteristic vector}, i.e.~the $i$-th coordinate of $v_F$ is $1$ if $i\in F$ and $0$ otherwise. In this way we can identify a set system $\cF \subseteq 2^{[n]}$ with the vector system
\begin{equation*}
\cV(\cF) = \Big\{ v_F\ :\ F\in \cF \Big\} \subseteq \{0,1 \}^n \subseteq \F^n,
\end{equation*}
where $\F$ is any field. Therefore, if there is an interesting problem about set systems, it is natural to embed the problem into $\F^n$ and ask what happens if we do not restrict ourselves to $0-1$ vectors but allow any set of coordinates. In case of shattering there is a usual way of generalizing the notion to general finite point sets (see e.g.~\cite{Sh}). Let $A\subseteq \F$ be a finite set of size $k$, and $\cV\subseteq A^n$ a finite point set whose elements we view for now as $[n]\rightarrow A$ functions. We say that $\cV$ \emph{shatters} a set $S\subseteq [n]$ if for every function $\ve g:S\rightarrow A$ there exists a function $\ve f \in \cV$ such that $\ve f|_S=\ve g$. As previously, let $\Sh(\cV)$ denote the family of sets shattered by $\cV$. In the definition of s-extremality the Sauer-Shelah lemma played a key role, however in this case we cannot expect a similar inequality to hold. Indeed, as $\Sh(\cV)\subseteq 2^{[n]}$, there are at most $2^n$ sets shattered, but at the same time the size of $\cV$ can be
much larger, up to $k^n$. Instead, we will use an algebraic characterization of s-extremal families \cite{M1,RM} to generalize the notion to arbitrary finite point sets. For this we will briefly review some of the necessary algebraic notions. For a more detailed introduction see Section \ref{sec:prelim} 

Given a finite point set $\cV\subseteq \F^n$ the \emph{vanishing ideal} $I(\cV) \unlhd \F[\textbf{x}]$\footnote{$\F[\textbf{x}]$ is the shorthand notation for the polynomial ring $\F[x_1,x_2,\dots,x_n]$.} is 
\begin{equation*}
I(\cV)= \Big\{ f\in \F[\textbf{x}]\ :\ f(\ve v)=0 \; \ \forall\; \ve v \in \cV \Big\}.
\end{equation*}
A total order $\prec$ on the monomials in $\F[\textbf{x}]$ is a \emph{term order}, if $1$ is the minimal element of $\prec$, and $\prec$ is compatible with multiplication with monomials. One well-known and important term order is the \emph{(standard) lexicographic order, lex order} for short,  where one has $\ve x^{\ve w} \prec \ve x^{\ve u}$\footnote{$\ve x^{\ve w}$ is the shorthand notation for the monomial $x_1^{w_1}x_2^{w_2}\cdots x_n^{w_n}$.} if and only if for the smallest index $i$ with $w_i \neq u_i$ one has $w_i < u_i$. One can build a lex order based on other orderings of the variables as well, so altogether we have $n!$ different lex orders on the monomials of $\Fxv$.

If $\prec$ is a term order and $f \in \F[\textbf{x}]$ a non-zero polynomial, then the \emph{leading monomial} $\text{lm}(f)$ of $f$ is the largest monomial (with respect to $\prec$) appearing with a non-zero coefficient in the canonical form of $f$. Given an ideal $I \unlhd \F[\textbf{x}]$ a monomial is called a \emph{standard monomial} of $I$ (with respect to $\prec$) if it is not a leading monomial of any $f\in I$. We denote by $\text{Lm}(I)$ and $\text{Sm}(I)$ the set of leading and standard monomials of $I$, respectively. In case of finite point sets the number of standard monomials of $I(\cV)$ is the same for every term order, it equals the size of $\cV$. Moreover, if $\cV=\cV(\cF)$ for some set system $\cF\subseteq 2^{[n]}$ then, as proved in \cite{M1,RM}, standard monomials also allow us to characterize s-extremal families.

\begin{prop}[\cite{M1,RM}]\label{extr<=>smua}
$\cF\subseteq 2^{[n]}$ is s-extremal if and only if the standard monomials of $I(\cV(\cF))$ are the same for every (lex) order.
\end{prop} 

Motivated by this result we define a finite point set $\cV\subseteq \F^n$ to be \emph{extremal} if $\text{Sm}(I(\cV))$ is the same for every lex order. The study of such extremal point sets was initiated in \cite{M1}. Some generalizations of results about s-extremal families, such as Theorem~\ref{Smdownshiftgen} and Corollary~\ref{downshiftcor}, were already proved here, but the general versions of the main results about set systems, namely Theorem~\ref{genextrGrobner} and Theorem~\ref{genalgothm}, were proven only in \cite{M2}.

At first sight this definition might seem rather artificial, but hopefully in the remainder of the article we will manage to convince the reader that extremal point sets are very natural generalizations of s-extremal set systems.

\subsubsection*{Main results}

Let $\cV$ be a finite point set. As we will see in Section~\ref{sec:prelim}, we may without loss of generality assume that $\m V\subseteq \{0,1,\dots,k-1\}^n\subseteq \mathbb{Q}^n$.
For $1\leq i\leq n$, the $i$-section of $\m V$ at the elements $\alpha_1,\dots,\alpha_{i-1},\alpha_{i+1},\dots,\alpha_n$ is defined as
\begin{equation*}
\m V_{i}(\alpha_1,\dots,\alpha_{i-1},\alpha_{i+1},\dots,\alpha_n)=\Big\{\alpha\ :\ (\alpha_1,\dots,\alpha_{i-1},\alpha,\alpha_{i+1},\dots,\alpha_n)\in \m V\Big\}.
\end{equation*}
Using $i$-sections one can define the \emph{downshift} of $\cV$ at coordinate $i$: $D_i(\m V)$ is the unique point set in $\{0,1,\dots,k-1\}^n$, for which
\begin{equation*}
(D_i(\m V))_i(\alpha_1,..,\alpha_{i-1},\alpha_{i+1},..,\alpha_n)=\Big\{0,1,..,|\m V_i(\alpha_1,..,\alpha_{i-1},\alpha_{i+1},..,\alpha_n)|-1\Big\}
\end{equation*}
whenever $\m V_i(\alpha_1,\dots,\alpha_{i-1},\alpha_{i+1},\dots,\alpha_n)$ is non-empty, and is empty otherwise. Note that when $k=2$ we recover the well-known and widely used downshift operation defined on set systems. For indices $i_1, i_2, \dots, i_{\ell}$ let 
\begin{equation*}
D_{i_1,i_2,\dots, i_{\ell}}(\cV):=D_{i_1}(D_{i_2}(\dots(D_{i_{\ell}}(\cV)))).
\end{equation*}
As demonstrated by the following theorem, downshifts are pretty handy when one wants to compute standard monomials.

\begin{thm}\label{Smdownshiftgen}
Let $\m V\subseteq\{0,1,\dots,k-1\}^n\subseteq \mathbb{Q}^n$ be a finite point set and $\prec$ the lex order order for which $x_{i_1}\succ x_{i_2}\succ \dots \succ x_{i_n}$. Then
\begin{equation*}
\text{Sm}(I(\cV))=D_{i_n,i_{n-1},\dots,i_1}(\cV).
\end{equation*}
\end{thm}

Theorem \ref{Smdownshiftgen} is a natural generalization of the same result for set systems (see Proposition~\ref{Smdownshift} below). 

According to Theorem~\ref{Smdownshiftgen} we could have defined extremal vector systems fully combinatorially, as demonstrated by the following corollary. 

\begin{cor}\label{downshiftcor}
A finite point set $\m V\subseteq\{0,1,\dots,k-1\}^n$ is extremal if and only if $D_{\pi(n),\pi(n-1),\dots,\pi(1)}(\cV)$ is the same for every permutation $\pi$ of $[n]$.
\end{cor}

We define a finite points set $\cV\subseteq \{0,1,\dots,k-1\}$ to be a \emph{down-set} (respectively \emph{up-set}), if $\ve v\in \cV$ and $\ve w\leq \ve v$ (respectively $\ve v\leq \ve w$) - coordinate-wise - imply $\ve w\in \cV$. Note that when $\cV$ is a down-set then $D_{i}(\cV)=\cV$ for every $i\in [n]$. Therefore Corollary~\ref{downshiftcor} also shows that, in accordance with the special case of set systems, down-sets are extremal. 

\bigskip

When working with a polynomial ideal, a nice ideal basis can facilitate things. An important example of such nice bases are Gr\"obner bases. For a term order $\prec$ and an ideal $I \unlhd \F[\textbf{x}]$ a finite subset $\mathbb{G} \subseteq I$ is called a \emph{Gr\"obner basis of I} with respect to $\prec$ if for every $f\in I$ there exists a $g\in \mathbb{G}$ such that $\text{Lm}(g)$ divides $\text{Lm}(f)$. $\mathbb{G}$ is a \emph{universal Gr\"obner basis} if it is a Gr\"obner basis for every term order. Gr\"obner bases have many nice properties, for details the interested reader may consult e.g.~\cite{AL}.

We call a polynomial $f(\ve x)\in\Fxv$ \emph{degree dominated} with \emph{dominating term} $\ve x^{\ve w}$ if it is of the form
\begin{equation*}
f(\ve x)=\ve x^{\ve w}+\sum\limits_{i=1}^{\ell}\alpha_i \ve x^{\ve v_i},
\end{equation*}
where $\ve x^{\ve v_i}\neq\ve x^{\ve w}$ and $\ve x^{\ve v_i}|\ve x^{\ve w}$ for every $i$. 

Just like s-extremal set systems, extremal point sets also admit a nice characterization using Gr\"obner bases. 

\begin{thm}\label{genextrGrobner}
A finite set of vectors $\cV\subseteq \{0,1,...,k-1\}^n$ is extremal if and only if there is a finite family $\mathcal{G}\subseteq \mathbb{R}[\ve
x]$ of degree dominated polynomials that form a universal \gr basis of $I(\cV)$.
\end{thm}

Theorem \ref{genextrGrobner} is a natural generalization of the same result for set systems (see Theorem~\ref{ExtremalGroebner} below). 

\bigskip

A term order $\prec$ is called an \emph{elimination order} with respect to the variable $x_i$ if $x_i$ is larger than any monomial from $\F[x_1 \dots,x_{i-1},x_{i+1},\dots,x_n]$. In particular in this case if $\ve u$ and $\ve v$ are two exponent vectors for which $u_i<v_i$, then necessarily $\ve x^{\ve u}\prec \ve x^{\ve v}$. As an example one can consider any lex order where $x_i$ is the largest variable.

\begin{thm}\label{genalgothm}
For $1\leq i\leq n$ let $\prec_i$ be an elimination order with respect to $x_i$. If $\cV\subseteq \{0,1,...,k-1\}^n$ is not extremal, then among these we can find two term orders for which the sets of standard monomials of $I(\cV)$ differ.
\end{thm}

Theorem~\ref{genalgothm} is a natural generalization of the same result for set systems (see Theorem~\ref{algthm} below) and it has several interesting consequences. 

First of all, it means that in the definition of extremality it would have been enough to require that the family of standard monomials is the same for a particular family of lex orders of size $n$. 

Next, as standard monomials can be computed in $O(n|\m V|k)$ time \cite{FRR}, Theorem~\ref{genalgothm}, just like in the special case of set systems, also results an efficient, $O(n^2|\m V|k)$ time algorithm for deciding whether a finite point set $\cV\subseteq
\{0,1,...,k-1\}^n$ is extremal or not. 

Finally, Theorem~\ref{genalgothm} also allows a strengthening of a result by Li, Zhang and Dong from \cite{LZD}, where they investigated the standard monomials of zero dimensional polynomial ideals. An ideal $I\triangleleft \Fxv$ is called \emph{zero dimensional} if the factor space $\Fxv/I$ is a finite dimensional $\F$-vector space. It is rather easy to see that vanishing ideals of finite point sets are special types of zero dimensional ideals. For $1\leq i\leq n$ let again $\prec_i$ be an elimination order with respect to $x_i$. Part $(2)\Leftrightarrow(3)$ of Theorem~4 in \cite{LZD} states that if $\F$ has characteristic zero, then the standard monomials of any zero dimensional ideal $I\triangleleft \Fxv$ are the same for every term order if and only if they are the same for $\prec_1,\dots,\prec_n$. The following result generalizes this to arbitrary fields instead of fields of characteristic zero.

\begin{thm}\label{LZDgen}
Let $\F$ be an arbitrary field and for $1\leq i\leq n$ let $\prec_i$ be an elimination order with respect to $x_i$. Then the standard monomials of any zero dimensional ideal $I\triangleleft \Fxv$ are the same for every term order if and only if they are the same for $\prec_1,\dots,\prec_n$.
\end{thm}

\smallskip

The rest of the paper is organized as follows. After the Introduction, in Section~\ref{sec:prelim} we give a more detailed description of the combinatorial and algebraic notion used. Then, In Section~\ref{sec:proofs} we give the proof of our main results. Finally, in Section~\ref{sec:concluding} we make some concluding remarks. 

\section{Preliminaries}\label{sec:prelim}

In this subsection we provide the reader with some details of the combinatorial and algebraic notions that appeared so far.

\subsubsection*{Vanishing ideals of set systems}

Let $\cF \subseteq 2^{[n]}$ and consider the vanishing ideal $I(\cF)=I(\cV(\cF)) \unlhd \F[\textbf{x}]$. Note that as $\cV(\cF)$ contains only $0-1$ vectors, we always have $\big\{x_i^2 -x_i\ :\ i\in [n] \big\} \subseteq I(\cF)$. In particular, $x_i^2$ is a leading monomial for every term order, implying that the standard monomials are necessarily square-free. By identifying a set $M\subseteq [n]$ with the square-free monomial $\ve x_M=\prod_{i\in M}x_i$ we can view $\text{Sm}(I(\cF))$ also as a subfamily of $2^{[n]}$. We will use these two views interchangeably and it will be always clear from the context which view is used.

\subsubsection*{Standard monomials}

A set $\mathcal{S}$ of monomials is called an \emph{up-set} (respectively \emph{down-set}) if $\ve x^{\ve u}\in \mathcal{S}$ and $\ve x^{\ve u}\mid \ve x^{\ve w}$ (respectively $\ve x^{\ve w}\mid \ve x^{\ve u}$) imply $\ve x^{\ve w}\in \mathcal{S}$. By the definition of ideals and leading monomials $\text{Lm}(I)$ is an up-set and hence $\text{Sm}(I)$, as its complement, is necessarily a down-set. Another basic fact about standard monomials is that the canonical image of $\text{Sm}(I)$ forms a linear basis of the $\F$-vector space $\F[\textbf{x}]/I$.

Moreover, in the special case when $I=I(\cV)$ for some finite point set $\cV\subseteq \F^n$ the factor space $\F[\textbf{x}]/I(\cV)$ is isomorphic to the space of all $\cV\rightarrow \F$ functions and as such has dimension $|\cV|$. This shows $|\text{Sm}(I(\cV))|=|\cV|$, as was already remarked in the Introduction.

A detailed study of the lex standard monomials of vanishing ideals of finite point sets was done by Felszeghy, R\'ath and R\'onyai in \cite{FRR}. Here we recall some of their results. Let $\cV\subseteq \F^n$ be a finite point set and for $\alpha\in \F$ let
\begin{equation*}
\cV_{\alpha}=\Big\{\ve v \in \F^{n-1}\ :\ (\ve v,\alpha)\in \cV\Big\}\subseteq \F^{n-1}.
\end{equation*}
Further let $\prec$ and $\prec_{\widehat{n}}$ be the standard lex order on $\F[x_1,\dots,x_n]$ and $\F[x_1,\dots,x_{n-1}]$, respectively. Then for $n>1$ the monomial $\ve x^{\ve w}$ is a standard monomial of $I(\cV)$ with respect to $\prec$ if and only if there are at least $w_n+1$ elements $\alpha\in \F$ for which $x_1^{w_1}\cdots x_{n-1}^{w_{n-1}}$ is a standard monomial of $I(\cV_{\alpha})$ with respect to $\prec_{\widehat{n}}$. Clearly, an analogous result holds for every possible lex order. As a base case for the above recursion, note that for $n=1$ if $\cV\subseteq \F$ then $x^w\in \text{Sm}(I(\cV))$ if and only if $w<|\cV|$. 

Using this recursion one can show that given any lex order $\text{Sm}(I(\cV))$ can be computed in linear, $O(n|\cV|k)$ time, where $k$ is the number of different coordinate values appearing in $\cV$.

Fix now some finite set $A\subseteq \F$ such that $\cV\subseteq A^n$. Note that as $\cV$ is finite, such an $A$ necessarily exists. Further let $\varphi_j:A\rightarrow \widehat{\F}$, $j=1,2,\dots,n$ be  injective functions to an arbitrary field $\widehat{\F}$. Denote by $\widehat{\cF}$ the image of $\m V$ under the action of $(\varphi_1,\dots,\varphi_n)$, i.e. 
\begin{equation*}
\widehat{\m V}=\Big\{(\varphi_1(v_1),\dots,\varphi_n(v_n))\ :\ (v_1,\dots,v_n)\in \m V\Big\}.
\end{equation*}
Then the standard monomials of $I(\cV)\triangleleft \Fxv$ are the same as those of $I(\widehat{\cV})\triangleleft \widehat{\F}[\ve x]$ for any lexicographic term order. We will refer to this property as the \emph{universality property of standard monomials}. An important corollary of this property is that it is enough to consider the case $\m V\subseteq \{0,1,\dots,k-1\}^n\subseteq \mathbb{Q}^n$.

\smallskip

In the case of finite set systems, apart from the above recursive description, lex standard monomials can also be computed using downshift operations. In this special case, which we recover by putting $k=2$ in the general definition, it is easier to describe the downshift operation in terms of set operations. For a set system $\cF\subseteq 2^{[n]}$ its downshift by the element $i\in [n]$ as
\begin{equation*}
D_i(\cF)=\Big\{F\setminus\{i\}\ :\ F\in \cF\Big\}~\bigcup~\Big\{F\ :\ F\in
\cF,\ i\in F,\ F\setminus\{i\}\in \cF\Big\}.
\end{equation*}
It is not hard to see that $|D_i(\m F)|=|\m F|$ and $\Sh(D_i(\m F))\subseteq \Sh(\m F)$, hence $D_i$ preserves s-extremality (see e.g.~\cite{BR}). As before, for indices $i_1,i_2,\dots, i_{\ell}$ we let $D_{i_1,i_2,\dots,i_{\ell}}(\cF):=D_{i_1}(D_{i_2}(\dots(D_{i_{\ell}}(\cF))))$. 

\begin{prop}[\cite{M1}]\label{Smdownshift}
Let $\cF \subseteq 2^{[n]}$ and $\prec$ be a lex order for which $x_{i_1}\succ x_{i_2}\succ \dots \succ x_{i_n}$. Then
\begin{equation*}
\text{Sm}(I(\cF))=D_{i_n,i_{n-1},\dots,i_1}(\cF).
\end{equation*}
\end{prop}

\smallskip

In what follows we demonstarte the connection between s-extremal families and standard monomials. The first key result in the algebraic characterization of s-extremal set systems was the description of the family of shattered sets using standard monomials. In \cite{M1,RM} it was proven that if $\cF\subseteq 2^{[n]}$ then
\begin{equation*}
\Sh(\cF)=\bigcup_{\text{all term orders}} \text{Sm}(I(\cF))=\bigcup_{\text{lex orders}} \text{Sm}(I(\cF)).
\end{equation*}
Since the number of standard monomials of $I(\cF)$ equals $|\cF|$ for every fixed term order, as a corollary we obtain Proposition~\ref{extr<=>smua} from the Introduction.

As mentioned earlier, for lex orders $\text{Sm}(I(\cF))$ can be computed in linear time, however the number of possible lex orders is $n!$, and so Proposition~\ref{extr<=>smua} does not offer directly an efficient method to check the s-extremality of a set system. However it turns out that we actually need only a significantly smaller collection of lex orders.

\begin{thm}[\cite{M1,RM}]\label{algthm}
Take $n$ orderings of the variables such that for every index $i$ there is one in which $x_i$ is the greatest element, and take the corresponding lex term orders. If $\cF\subseteq 2^{[n]}$ is not s-extremal, then among these we can find two term orders for which the sets of standard monomials of $I({\cF})$ differ.  
\end{thm}

Accordingly, by computing the standard monomials for $n$ lex orders we can decide the extremality of a set system in $O(n^2|\cF|)$ time.

\subsubsection*{Gr\"obner bases}

The algebraic results about s-extremal families also include a nice connection between s-extremal families and the theory of Gr\"obner bases. Given a pair of sets $H \subseteq S \subseteq [n]$ we define the polynomial 
\begin{equation*}
f_{S,H}(\textbf{x}) = \textbf{x}_H \cdot \prod_{i \in S\setminus H}(x_i -1). 
\end{equation*}
A useful property of these polynomials is that for a set $F\subseteq [n]$ we have $f_{S,H}(v_F) \neq 0$ if and only if $F\cap S = H$. However, actually much more is true.
\begin{thm}[\cite{M1},\cite{RM}]\label{ExtremalGroebner}
$\cF \subseteq 2^{[n]}$ is s-extremal if and only if there are polynomials of the form $f_{S,H}$, which together with $\big\{x_i^2-x_i\; : \; i \in [n] \big\}$ form a universal Gr\"obner basis of $I(\cF)$. 
\end{thm}
We remark that in Theorem~\ref{ExtremalGroebner} it is enough to require a Gr\"obner basis of the above form for just one term order to have an s-extremal family. Also note that all the polynomials appearing in Theorem~\ref{ExtremalGroebner} are degree dominated. 

\section{Proofs}\label{sec:proofs}

Before proving our main results, we consider a rather technical statement that is needed to guarantee that the definition of extremality in this general setting is compatible with the special case of set systems. 

\begin{prop}\label{lex<=>all term orders}
Let $A\subseteq \F$ be a finite set of size $k$ and $\cV\subseteq A^n$ a finite point set. Then $\text{Sm}(I(\cV))$ is the same for every lexicographic term order if and only if $\text{Sm}(I(\cV))$ is the same for every term order.
\end{prop}
\begin{proof}
One direction is just trivial. For the other direction suppose that the standard monomials of $I(\cV)$ are the same for every lex order, and denote this collection of monomials by $\mathcal{S}$. Now take an arbitrary monomial $\ve x^{\ve u}\notin \m S$. Then $\ve x^{\ve u}$ is a leading monomial with respect to every lex order. Fix one lex order. Since standard monomials form a linear basis of the $\F$-vector space $\F[\textbf{x}]/I(\cV)$, each leading monomial, in particular $\ve x^{\ve u}$, has a representation by standard monomials, i.e. there are coefficients $\alpha_{\ve v}$, $\ve x^{\ve v}\in \mathcal{S}$, such that 
\begin{equation*}
f(\ve x)=\ve x^{\ve u}+\sum\limits_{x^{\ve v}\in \m S}\alpha_{\ve v} x^{\ve v}\in I(\cV).
\end{equation*}
As $\m S$ is the family of standard monomials for every lex order, this representation is necessarily valid for any other lex order as well, in particular, the leading monomial of $f$ is $\ve x^{\ve u}$ for any other lex order as well. This is possible only if $f(\ve x)$ is a degree dominated polynomial with dominating term $\ve x^{\ve u}$. Indeed suppose this is not the case, and there is some monomial $x^{\ve v}\in \m S$  that appears with a nonzero coefficient in $f$ and $x^{\ve v}\nmid x^{\ve u}$. For this there has to be an index $i$ for which $v_i>u_i$, but then for any lex order where $x_i$ is the largest variable we would have that $x^{\ve u}\prec x^{\ve v}$, contradicting our assumption that $\text{lm}(f)=x^{\ve u}$.

On the other hand, as any term order is compatible with multiplication by monomials, the dominating term of a degree dominated polynomial is necessarily its leading monomial for every term order. Therefore the leading monomial of $f$ is $x^{\ve u}$ for every term order. This results, that every monomial $\ve x^{\ve u}\notin \m S$ is a leading monomial for every term order. Adding the fact that the number of standard monomials, i.e. the number of non-leading monomials, is
$|\cV|$ for every term order, we get that the family of leading monomials, and hence the family of standard monomials of $I(\cV)$ is the same for every term order as desired.
\end{proof}

According to Proposition~\ref{lex<=>all term orders} it does not matter whether we define a finite set of points $\cV\subseteq A^n\subseteq \F^n$ to be extremal if $\text{Sm}(I(\cV))$ is the same only for every lex order, or for every term order. 

\subsubsection*{Downshift operations}

\begin{proof}[Proof of Theorem~\ref{Smdownshiftgen}]
Without loss of generality we may assume that $\prec$ is the standard lex order, i.e. $i_j = j$ for every $j$. To prove the statement we apply induction on $n$. For $n = 1$, $x^w\in \text{Sm}(I(\cV))$ and $w \in D_1(\cV)$ hold at the same time, namely when $w < |\cV|$, thus for $n = 1$ the statement holds. Now consider the case $n > 1$ and suppose that the statement is true for all values smaller than $n$. For $i\neq n$ the downshift
$D_i$ acts independently on the subsystems $\cV_{\beta}$, $\beta\in \{0,1,\dots,k-1\}$, namely we have
\begin{equation*}
D_{n-1,\dots,1}(\cV) = \bigcup_{\beta=0}^{k-1}\Big\{(\ve w,\beta)\ :\ \ve w\in D_{n-1,\dots,1}(\cV_{\beta})\Big\}.
\end{equation*}
According to the induction hypothesis we have $\text{Sm}(I(\cV_{\beta}))=D_{n-1,\dots,1}(\cV_{\beta})$  for every $\beta\in \{0,1,\dots,k-1\}$. However, then, when constructing $\text{Sm}(I(\cV))$ from the $\text{Sm}(I(\cV_{\beta}))$-s we apply the same rule for the exponent vectors as for the elements of $D_{n,n-1,\dots,1}(\cV)$ when constructing it from the $D_{n-1,\dots,1}(\cV_{\beta})$-s. More precisely we have $x^{(\ve w,w_n)} \in \text{Sm}(I(\cV))$ exactly in the same case when $(\ve w, w_n) \in D_{n,n-1,\dots,1}(\cV)$, namely when there are at least $w_n+1$ values $\beta\in \{0,1,\dots,k-1\}$ such that $\ve x^{\ve w} \in \text{Sm}(I(\cV_{\beta}))$  \big($\ve w\in D_{n-1,\dots,1}(\cV_{\beta})$\big). This finishes the proof.
\end{proof}

\subsubsection*{Gr\"obner characterization}

\begin{proof}[Proof of Theorem~\ref{genextrGrobner}]
First suppose that $\cV\subseteq \{0,1,...,k-1\}^n$ is extremal. Then by definition, $\text{Sm}(I(\cV))$ is the same for every term order. Denote this set of monomials by $\mathcal{S}$ and the collection of all minimal (with respect to division) monomials not belonging to $\mathcal{S}$ by $\mathbb{S}$. Note that in this case the family of leading monomials is also the same for every term order and $\mathbb{S}$ is just the set of minimal elements there. In particular for every leading monomial $\ve x^{\ve u}$ there is another leading monomial $\ve x^{\ve v}\in \mathbb{S}$ such that $\ve x^{\ve v}\mid\ve x^{\ve u}$.

From this point we follow the line of thinking from the proof of Proposition~\ref{lex<=>all term orders}. Each monomial $\ve x^{\ve u}\in \mathbb{S}$ is a leading monomial for every term order, in particular for the standard lex order as well. Accordingly, as before, $\ve x^{\ve u}$ has a representation by standard monomials with respect to the standard lex order, i.e. there are coefficients $\alpha_{\ve v}$, $\ve x^{\ve v}\in \mathcal{S}$ such that 
\begin{equation*}
f_{\ve u}(\ve x)=\ve x^{\ve u}+\sum\limits_{\ve x^{\ve v}\in \mathcal{S}}\alpha_{\ve v} x^{\ve v}\in I(\cV).
\end{equation*}
As by assumption $\text{Sm}(I(\cV))=\mathcal{S}$
for every term order, and except of $x^{\ve u}$ every monomial in $f_{\ve u}(\ve x)$ is a standard one, we necessarily have that $\text{lm}(f_{\ve u})=x^{\ve u}$ for every term order. However, in the same way as in the proof of Proposition~\ref{lex<=>all term orders}, this is possible only if $f_{\ve u}$ is degree dominated with dominating term $x^{\ve u}$. Put now
\begin{equation*}
\mathcal{G}=\Big\{f_{\ve u}\ :\ \ve x^{\ve u}\in \mathbb{S}\Big\}.
\end{equation*}
By the definition of $\mathbb{S}$ the family $\mathcal{G}$ is clearly a \gr basis of $I(\cV)$ for every term order, i.e. it is a universal \gr basis.

For the other direction suppose that we are given a finite family $\mathcal{G}\subseteq \F[\ve x]$ of degree dominated polynomials that form a universal \gr basis of $I(\cV)$. We prove that the fact that there is a common \gr basis for every term order, without knowing anything about the members of the \gr basis, already guarantees the extremality of $\cV$.

If we are given a \gr basis $\mathcal{G}$ of some ideal $I$ for a fixed term order, it determines $\text{Lm}(I)$ uniquely, and hence $\text{Sm}(I)$ as well, namely we have
\begin{equation*}
\text{Lm}(I)=\Big\{\ve x^{\ve u}\ :\ \exists g\in \mathcal{G} \mbox{ such that }\text{lm}(g)|\ve x^{\ve u}\Big\}.
\end{equation*}
Indeed, the containment in one direction follows from the definition of \gr bases. For the other direction note that if for some $g\in \mathcal{G}$ we have that $\text{lm}(g)|\ve x^{\ve u}$, then the polynomial $\frac{\ve x^{\ve u}}{\text{lm}(g)}g(\ve x)\in I$ shows that $\ve x^{\ve u}\in \text{Lm}(I)$.

However as $\mathcal{G}$ is a common \gr basis for every term order, it gives us the same family of standard monomials for every term order, and so the extremality of $\cV$ follows.
\end{proof}

We remark that similarly as in the case of Theorem~\ref{ExtremalGroebner}, in Theorem~\ref{genextrGrobner} it is also enough to require that $I(\cV)$ has a suitable \gr basis for some term order. 

\subsubsection*{Testing extremality}

\begin{proof}[Proof of Theorem~\ref{genalgothm}]
By contraposition it is enough to prove that if the standard monomials of $I(\cV)$ are the same for the above term orders, then $\cV$ is extremal. Accordingly suppose the condition holds, and denote the collection of standard monomials for the above term orders by $\mathcal{S}$. From now on we again follow the proof of Proposition~\ref{lex<=>all term orders}. Take an arbitrary monomial $\ve x^{\ve u}\notin \m S$. In this case $\ve x^{\ve u}$ is a leading monomial with respect to all of the $n$ term orders considered. Fix one of these term orders, and take the standard representation of $\ve x^{\ve u}$ with respect to this term order. 
\begin{equation*}
f(\ve x)=\ve x^{\ve u}+\sum\limits_{x^{\ve v}\in \m S}\alpha_{\ve v} x^{\ve v}\in I(\cV).
\end{equation*}
As $\m S$ is the family of standard monomials for the other $n-1$ term orders as well, the leading monomial of $f$ can be only $\ve x^{\ve u}$ for them as well, and then the same reasoning as before shows that this is possible only if $f$ is degree dominated with dominating term $x^{\ve u}$. Indeed, suppose
this is not the case, and there is some monomial $x^{\ve v}\in \m S$  that appears with a nonzero coefficient in $f$ and $x^{\ve v}\nmid x^{\ve u}$. Now there has to be an index $i$ for which $v_i>u_i$, but then for any elimination order with respect to $x_i$, in particular for the one in our collection, we would have that $x^{\ve u}\prec x^{\ve v}$, contradicting our assumption that $\text{lm}(f)=x^{\ve u}$ for that lex order. From this the extremality of $\cV$ follows exactly as in Proposition~\ref{lex<=>all term orders}.
\end{proof}

\subsubsection*{Zero dimensional ideals}

\begin{proof}[Proof of Theorem~\ref{LZDgen}]
The proof follows from Theorem~\ref{genalgothm} together with the universality property of standard monomials. For this just note that the proof of Theorem~\ref{genalgothm} does not use that $I(\cV)$ is a vanishing ideal, the fact that it is a zero dimensional ideal suffices.
\end{proof}

\section{Concluding remarks}\label{sec:concluding}

Given the algebraic characterization of s-extremal families, extremal points sets are a natural generalization, both from the combinatorial and from the algebraic point of view. As a next step it would be interesting to obtain further notable examples of such point sets and explore their combinatorial properties. 

In connection with Theorem~\ref{ExtremalGroebner} and Theorem~\ref{genextrGrobner} recall that in the case $k=2$ we have a more precise description of the degree dominated polynomials appearing in the \gr basis. Can we obtain a similar description in the general case? For this one could start with studying the family of pairs of sets appearing in Theorem~\ref{ExtremalGroebner} and try to understand under which conditions do the resulting polynomials together with the polynomials of the form $x_i^2-x_i$ form a \gr basis. 

Theorem~\ref{LZDgen} demonstrates that researchers from the algebra community are also interested in related problems. In this direction examples of extremal zero dimensional ideals, that are not vanishing ideals could help us to understand better the algebraic properties behind these structures.

\end{document}